\documentclass[12pt,twoside]{amsart}
\usepackage{amsmath}
\usepackage{amsthm}
\usepackage{amsfonts}
\usepackage{amssymb}
\usepackage{latexsym}
\usepackage{mathrsfs}
\usepackage{amsmath}
\usepackage{amsthm}
\usepackage{amsfonts}
\usepackage{amssymb}
\usepackage{latexsym}
\usepackage{geometry}
\usepackage{dsfont}
\usepackage[dvips]{graphicx}
\usepackage{color}
\usepackage[all]{xy}

\date{}
\allowdisplaybreaks[4] \footskip=15pt
\renewcommand{\uppercasenonmath}[1]{}

\usepackage{graphicx,amssymb}
\usepackage[all]{xy}
\usepackage{amsmath}

\allowdisplaybreaks
\usepackage{amsthm}
\usepackage{color}

\theoremstyle{plain}
\newtheorem{theorem}{Theorem}[section]
\newtheorem{proposition}[theorem]{Proposition}
\newtheorem{lemma}[theorem]{Lemma}
\newtheorem{corollary}[theorem]{Corollary}
\newtheorem{example}[theorem]{Example}
\newtheorem*{open question}{Open Question}
\newtheorem{definition}[theorem]{Definition}

\theoremstyle{definition}
\newtheorem*{acknowledgement}{Acknowledgement}

\theoremstyle{remark}

\newcommand{\C}{\mathcal{C}}

\newcommand{\Id}{\mathrm{Id}}

\def\p{\frak p}
\def\m{\frak m}

\def\Hom{{\rm Hom}}
\def\Ext{{\rm Ext}}

\def\Ker{{\rm Ker}}

\def\Im{{\rm Im}}
\def\Coker{{\rm Coker}}

\def\Max{{\rm Max}}

\def\Spec{{\rm Spec}}
\def\Max{{\rm Max}}

\def\Id{{\rm Id}}

\begin{document}
\begin{center}
{\large  \bf The $S$-global dimensions of commutative rings}

\vspace{0.5cm}   Xiaolei Zhang$^{a}$,\ Wei Qi$^{b}$

{\footnotesize  Department of Basic Courses, Chengdu Aeronautic Polytechnic, Chengdu 610100, China\\
b.\ School of Mathematical Sciences, Sichuan Normal University, Chengdu 610068, China\\

E-mail: zxlrghj@163.com\\}
\end{center}

\bigskip
\centerline { \bf  Abstract}
\bigskip
\leftskip10truemm \rightskip10truemm \noindent

Let $R$ be a commutative ring with identity and $S$  a multiplicative subset of $R$. First, we introduce and study the  $S$-projective dimensions and $S$-injective dimensions of $R$-modules, and then explore the $S$-global dimension $S$-gl.dim$(R)$ of a commutative ring $R$ which is defined to be  the supremum of $S$-projective dimensions of all $R$-modules. Finally, we investigated the $S$-global dimension of factor rings and  polynomial rings.
\vbox to 0.3cm{}\\
{\it Key Words:}    $S$-projective dimensions, $S$-injective dimensions,  $S$-global dimensions, polynomial rings.\\
{\it 2010 Mathematics Subject Classification:}   13D05, 13D07.

\leftskip0truemm \rightskip0truemm
\bigskip

Throughout this article, $R$ always is a commutative ring with identity and $S$ always is a multiplicative subset of $R$, that is, $1\in S$ and $s_1s_2\in S$ for any $s_1\in S, s_2\in S$. In 2002,  Anderson and Dumitrescu \cite{ad02} defined $S$-Noetherian rings $R$ for which any ideal of $R$ is $S$-finite. Recall from \cite{ad02} that an $R$-module $M$ is called $S$-finite provided that $sM\subseteq F$ for some $s\in S$ and some finitely generated submodule $F$ of $M$. An $R$-module $T$ is called uniformly $S$-torsion if $sT=0$ for some $s\in S$ (see \cite{zwz21}). So  an $R$-module $M$ is  $S$-finite if and only if $M/F$ is uniformly $S$-torsion for some finitely generated submodule $F$ of $M$. The idea derived from  uniformly $S$-torsion modules  is deserved to be further investigated.

 In \cite{zwz21-p}, the author of this paper introduced the class of $S$-projective modules $P$ for which the functor $\Hom_R(P,-)$ preserves $S$-exact sequences. The class of  $S$-projective modules can be seen as a ``uniform'' generalization of that of projective modules, since an $R$-module $P$ is $S$-projective if and only if  $\Ext_R^1(P,M)$ is  uniformly $S$-torsion for any  $R$-module $M$ (see \cite[Theorem 2.5]{zwz21-p}). The class of  $S$-projective modules owns the following $S$-hereditary property: let $0\rightarrow A\xrightarrow{f} B\xrightarrow{g} C\rightarrow 0$ be an  $S$-exact sequence, if $B$ and $C$ are  $S$-projective so is $A$ (see \cite[Proposition 2.8]{zwz21-p}). So it is worth to study the  $S$-analogue of projective dimensions of $R$-modules. Similarly, By the discussion of $S$-injective modules in \cite{QKWCZ21}, we can study  the  $S$-analogue of injective dimensions of $R$-modules. Together these, an  $S$-analogue of global dimensions of commutative rings can also be introduced and studied.

In this article, we define the $S$-projective dimension $S$-$pd_R(M)$ (resp., $S$-injective dimension $S$-$id_R(M)$) of an  $R$-module $M$ to be the length of the shortest $S$-projective (resp.,  $S$-injective) $S$-resolution of $M$.  We characterize  $S$-projective dimensions (resp.,  $S$-injective) of $R$-modules using the  uniform torsion property of  the ``Ext''  functors  in Proposition  \ref{s-projective d} (resp., Proposition  \ref{s-inj-d}). Besides,  we obtain local characterizations of projective dimensions and injective dimensions of $R$-modules in Corollary \ref{wgld-swgld}.  The $S$-global dimension $S$-gl.dim$(R)$ of a commutative ring $R$  is defined to be  the supremum of $S$-projective dimensions of all $R$-modules.  We find that $S$-global dimensions of commutative rings is also the supremum of $S$-injective dimensions of all $R$-modules. A new characterization of global dimensions is given in Corollary \ref{wgld-swgld}. $S$-semisimple rings are firstly introduced in \cite{zwz21-p}  for which  any free  $R$-module is $S$-semisimple. By \cite[Theorem 3.11]{zwz21}, a ring $R$ is $S$-semisimple if and only if all $R$-modules are $S$-projective (resp.,  $S$-injective). So $S$-semisimple  are exactly commutative rings with $S$-global dimension equal to $0$ (see Corollary \ref{s-vn-ext-char}).  In the final section, we investigate the $S$-global dimensions of factor rings and  polynomial rings and show that $S$-gl.dim$(R[x])=S$-gl.dim$(R)+1$ (see Theorem \ref{s-wgd-poly}).

\section{Preliminaries}

Recall from \cite{zwz21}, an $R$-module $T$ is called a uniformly $S$-torsion module  provided that there exists an element $s\in S$ such that $sT=0$.
An $R$-sequence  $M\xrightarrow{f} N\xrightarrow{g} L$ is called  \emph{$S$-exact} (at $N$) provided that there is an element $s\in S$ such that $s\Ker(g)\subseteq \Im(f)$ and $s\Im(f)\subseteq \Ker(g)$. We say a long $R$-sequence $...\rightarrow A_{n-1}\xrightarrow{f_n} A_{n}\xrightarrow{f_{n+1}} A_{n+1}\rightarrow...$ is $S$-exact, if for any $n$ there is an element $s\in S$ such that $s\Ker(f_{n+1})\subseteq \Im(f_n)$ and $s\Im(f_n)\subseteq \Ker(f_{n+1})$. An $S$-exact sequence $0\rightarrow A\rightarrow B\rightarrow C\rightarrow 0$ is called a short $S$-exact sequence. Let $\xi: 0\rightarrow A\xrightarrow{f} B\xrightarrow{g} C\rightarrow 0$ be an $S$-short exact sequence. Then $\xi$ is said to be $S$-split provided that there is  $s\in S$ and $R$-homomorphism $f':B\rightarrow A$ such that $f'(f(a))=sa$ for any $a\in A$, that is, $f'\circ f=s\Id_A$ (see \cite[Definition 2.1]{zwz21-p}).

An $R$-homomorphism $f:M\rightarrow N$ is an \emph{$S$-monomorphism}  $($resp.,   \emph{$S$-epimorphism}, \emph{$S$-isomorphism}$)$ provided $0\rightarrow M\xrightarrow{f} N$   $($resp., $M\xrightarrow{f} N\rightarrow 0$, $0\rightarrow M\xrightarrow{f} N\rightarrow 0$ $)$ is   $S$-exact.
It is easy to verify an  $R$-homomorphism $f:M\rightarrow N$ is an $S$-monomorphism $($resp., $S$-epimorphism, $S$-isomorphism$)$ if and only if $\Ker(f)$ $($resp., $\Coker(f)$, both $\Ker(f)$ and $\Coker(f)$$)$ is a  uniformly $S$-torsion module. Let $R$ be a ring and  $S$ a multiplicative subset of $R$. Suppose $M$ and $N$ are $R$-modules. We say $M$ is $S$-isomorphic to $N$ if there exists an $S$-isomorphism $f:M\rightarrow N$. A family $\C$  of $R$-modules  is said to be closed under $S$-isomorphisms if $M$ is $S$-isomorphic to $N$ and $M$ is in $\C$, then $N$ is  also in  $\C$. It follows from \cite[Proposition 1.1]{z21-swd} that the existence of $S$-isomorphisms of two $R$-modules is actually an equivalence relation.

\begin{lemma}\cite[Proposition 1.1]{z21-swd}\label{s-iso-inv}
Let $R$ be a ring and $S$ a multiplicative subset of $R$. Suppose there is an $S$-isomorphism $f:M\rightarrow N$ for  $R$-modules  $M$ and $N$. Then there is an $S$-isomorphism $g:N\rightarrow M$. Moreover, there is $s\in S$ such that $f\circ g=s\Id_N$ and $g\circ f=s\Id_M$.
\end{lemma}

The following result says that a short $S$-exact sequence induces a long $S$-exact sequence by the functor ``Ext'' as the classical case.

\begin{lemma}\label{s-iso-ext}
Let $R$ be a ring and  $S$ a multiplicative subset of $R$. Let $L$, $M$ and $N$ be $R$-modules. If $f:M\rightarrow N$ is an $S$-isomorphism, then $\Ext^{n}_R(L,f):\Ext^{n}_R(L,M)\rightarrow \Ext^{n}_R(L,N)$  and $\Ext^{n}_R(f,L):\Ext^{n}_R(N,L)\rightarrow \Ext^{n}_R(M,L)$ are all $S$-isomorphisms for any $n\geq 0$.
\end{lemma}
\begin{proof} We only show  $\Ext^{n}_R(L,f):\Ext^{n}_R(L,M)\rightarrow \Ext^{n}_R(L,N)$ is an $S$-isomorphism for any $n\geq 0$ since the other one is similar. Consider the exact sequences: $0\rightarrow \Ker(f)\rightarrow M\xrightarrow{\pi_{\Im(f)}} \Im(f)\rightarrow 0$ and  $0\rightarrow \Im(f)\xrightarrow{i_{\Im(f)}} N\rightarrow \Coker(f)\rightarrow 0$ with $\Ker(f)$ and $\Coker(f$ uniformly $S$-torsion. Then there are long exact sequences $$\Ext^{n}_R(L,\Ker(f))\rightarrow \Ext^{n}_R(L,M)\xrightarrow{\Ext^{n}_R(L,\pi_{\Im(f)})} \Ext^{n}_R(L,\Im(f))\rightarrow \Ext^{n+1}_R(L,\Ker(f))$$ and $$\Ext^{n-1}_R(L,\Coker(f))\rightarrow \Ext^{n}_R(L,\Im(f))\xrightarrow{\Ext^{n}_R(L,i_{\Im(f)})} \Ext^{n}_R(L,N)\rightarrow \Ext^{n}_R(L,\Coker(f)).$$
Since $\Ext^{n}_R(L,\Ker(f))$, $ \Ext^{n+1}_R(L,\Ker(f))$, $\Ext^{n-1}_R(L,\Coker(f))$  and $\Ext^{n}_R(L,\Coker(f))$ are all uniformly $S$-torsion by \cite[Lemma 4.2]{QKWCZ21}, we have  $$\Ext^{n}_R(L,f):\Ext^{n}_R(L,M)\xrightarrow{\Ext^{n}_R(L,\pi_{\Im(f)})} \Ext^{n}_R(L,\Im(f))\xrightarrow{\Ext^{n}_R(L,i_{\Im(f)})} \Ext^{n}_R(L,N)$$ is an $S$-isomorphism.
\end{proof}

\begin{theorem}\label{s-iso-tor}
Let $R$ be a ring, $S$ a multiplicative subset of $R$ and $M$ and $N$ $R$-modules. Suppose $0\rightarrow A\xrightarrow{f} B\xrightarrow{g} C\rightarrow 0$ is an $S$-exact sequence of $R$-modules. Then for any $n\geq 1$ there is an $R$-homomorphism $\delta_n:\Ext^{n-1}_R(M,C)\rightarrow \Ext^{n}_R(M,A)$ such that  the induced sequences
\begin{center}
$0\rightarrow \Hom_R(M,A)\rightarrow \Hom_R(M,B)\rightarrow \Hom_R(M,C)\rightarrow \Ext^{1}_R(M,A)\rightarrow \cdots \rightarrow\Ext^{n-1}_R(M,B)\rightarrow \Ext^{n-1}_R(M,C)\xrightarrow{\delta_n} \Ext^{n}_R(M,A)\rightarrow \Ext^{n}_R(M,B)\rightarrow \cdots$
\end{center}
and
\begin{center}
$0\rightarrow \Hom_R(C,N)\rightarrow \Hom_R(B,N)\rightarrow \Hom_R(A,N)\rightarrow \Ext^{1}_R(C,N)\rightarrow \cdots \rightarrow
 \Ext^{n-1}_R(B,N)\rightarrow \Ext^{n-1}_R(A,N)\xrightarrow{\delta_n} \Ext^{n}_R(C,N)\rightarrow  \Ext^{n}_R(B,N)\rightarrow \cdots$
 \end{center}
 are $S$-exact.
\end{theorem}
\begin{proof}
We only show the first sequence is $S$-exact  since the other one is similar. Since the sequence $0\rightarrow A\xrightarrow{f} B\xrightarrow{g} C\rightarrow 0$ is $S$-exact at $B$. There is an exact sequence $0\rightarrow \Ker(g)\xrightarrow{i_{\Ker(g)}} B\xrightarrow{\pi_{\Im(g)}} \Im(g)\rightarrow 0$. . So There is a long exact sequence of $R$-modules:
 \begin{center}
$0\rightarrow \Hom_R(M,\Ker(g))\rightarrow \Hom_R(M,B)\rightarrow \Hom_R(M,\Im(g))\rightarrow \Ext^{1}_R(M,\Ker(g))\rightarrow \cdots \rightarrow\Ext^{n-1}_R(M,B)\rightarrow \Ext^{n-1}_R(M,\Im(g))\xrightarrow{\delta'_n} \Ext^{n}_R(M,\Ker(g))\rightarrow \Ext^{n}_R(M,B)\rightarrow \cdots$
 \end{center}
Note that there are $S$-isomorphisms $t_1:A\rightarrow \Ker(g)$, $t'_1:\Ker(g)\rightarrow A$, $t_2: \Im(g)\rightarrow C$ and $t'_2: C\rightarrow \Im(g)$ by Lemma \ref{s-iso-inv}. So, by Lemma \ref{s-iso-ext},  
$\Ext^{n}_R(M,t'_1): \Ext^{n}_R(M,\Ker(g))\rightarrow\Ext^{n}_R(M,A)$
and  $\Ext^{n}_R(M,t'_2): \Ext^{n}_R(M,C)\rightarrow\Ext^{n}_R(M,\Im(g))$  are $S$-isomorphisms for any $n\geq 0$. Setting $\delta_n= \Ext^{n}_R(M,t'_1)\circ\delta'_n\circ\Ext^{n}_R(M,t'_2)$, we have an $S$-exact sequence
\begin{center}
$0\rightarrow \Hom_R(M,A)\rightarrow \Hom_R(M,B)\rightarrow \Hom_R(M,C)\rightarrow \Ext^{1}_R(M,A)\rightarrow \cdots \rightarrow\Ext^{n-1}_R(M,B)\rightarrow \Ext^{n-1}_R(M,C)\xrightarrow{\delta_n} \Ext^{n}_R(M,A)\rightarrow \Ext^{n}_R(M,B)\rightarrow \cdots$
\end{center}
\end{proof}

Recall from \cite[Definition 3.1]{zwz21-p} that an $R$-module $P$ is called $S$-projective provided that the induced sequence $$0\rightarrow \Hom_R(P,A)\rightarrow \Hom_R(P,B)\rightarrow \Hom_R(P,C)\rightarrow 0$$ is $S$-exact for any $S$-exact sequence $0\rightarrow A\rightarrow B\rightarrow C\rightarrow 0$. And recall from \cite[Definition 4.1]{QKWCZ21} that an $R$-module $E$ is called $S$-injective provided that the induced sequence
 $$0\rightarrow \Hom_R(C,E)\rightarrow \Hom_R(B,E)\rightarrow \Hom_R(A,E)\rightarrow 0$$
 is $S$-exact for any $S$-exact sequence $0\rightarrow A\rightarrow B\rightarrow C\rightarrow 0$. Following from \cite[Theorem 3.2]{zwz21}, an  $R$-module $P$ is projective if and only if  $\Ext_R^1(P,M)$ is  uniformly $S$-torsion for any  $R$-module $M$, if and only if every $S$-short exact sequence $0\rightarrow A\xrightarrow{f} B\xrightarrow{g} P\rightarrow 0$ is $S$-split. Similarly, an $R$-module $E$  is $S$-injective if and only if $\Ext_R^1(M,E)$ is  uniformly $S$-torsion for any $R$-module $M$, if and only  every $S$-short exact sequence $0\rightarrow E\xrightarrow{f} A\xrightarrow{g} B\rightarrow 0$ is $S$-split by \cite[Theorem 4.3]{QKWCZ21} and \cite[Proposition 2.3]{zwz21-p}. Following  from Theorem \ref{s-iso-tor}, we have the following result.

\begin{corollary}\label{big-Tor}
Let $R$ be a ring, $S$ a multiplicative subset of $R$ and $M$ and $N$  $R$-modules. Suppose $0\rightarrow A\xrightarrow{f} B\xrightarrow{g} C\rightarrow 0$ is an $S$-exact sequence of $R$-modules.

 \begin{enumerate}
    \item If $B$ is $S$-projective, then $\Ext^R_{n}(C,N)$ is $S$-isomorphic to  $\Ext_{n+1}^R(A,N)$ for any $n\geq 0$.
    \item  If $B$ is $S$-injective, then $\Ext^R_{n}(M,A)$ is $S$-isomorphic to  $\Ext_{n+1}^R(M,C)$ for any $n\geq 0$.
 \end{enumerate}
\end{corollary}

\section{On the $S$-projective dimensions and $S$-injective dimensions of modules}

In this section we mainly introduced the the $S$-versions of projective dimensions and injective dimensions of $R$-modules.

\begin{definition}\label{s-prd }
Let $R$ be a ring, $S$ a multiplicative subset of $R$  and $M$ an $R$-module. We write $S$-$pd_R(M)\leq n$  $(S$-$pd$ abbreviates  \emph{$S$-projective dimension}$)$ if there exists an $S$-exact sequence of $R$-modules
$$ 0 \rightarrow F_n \rightarrow ...\rightarrow F_1\rightarrow F_0\rightarrow M\rightarrow 0   \ \ \ \ \ \ \ \ \ \ \ \ \ \ \ \ \ \ \ \ \ \ \ \ \ \ \ \ \ \ \ \ \ \ \ \ \ \ \ (\diamondsuit)$$
where each $F_i$ is  $S$-projective for $i=0,...,n$. The $S$-exact sequence $(\diamondsuit)$ is said to be an $S$-projective $S$-resolution of length $n$ of $M$. If such finite  $S$-projective  $S$-resolution does not exist, then we say $S$-$pd_R(M)=\infty$; otherwise,  define $S$-$pd_R(M)=n$ if $n$ is the length of the shortest $S$-projective $S$-resolution of $M$.
\end{definition}\label{def-wML}

\begin{definition}\label{s-injd }
Let $R$ be a ring, $S$ a multiplicative subset of $R$  and $M$ an $R$-module. We write $S$-$id_R(M)\leq n$  $(S$-$id$ abbreviates  \emph{$S$-injective dimension}$)$ if there exists an $S$-exact sequence of $R$-modules
$$ 0 \rightarrow M \rightarrow E_0\rightarrow E_1 ...\rightarrow E_{n-1}\rightarrow E_{n}\rightarrow 0   \ \ \ \ \ \ \ \ \ \ \ \ \ \ \ \ \ \ \ \ \ \ \ \ \ \ \ \ \ \ \ \ \ \ \ \ \ \ \ (\star)$$
where each $E_i$ is  $S$-injective for $i=0,...,n$. The $S$-exact sequence $(\star)$ is said to be an $S$-injective $S$-resolution of length $n$ of $M$. If such finite  $S$-injective  $S$-resolution does not exist, then we say $S$-$id_R(M)=\infty$; otherwise,  define $S$-$id_R(M)=n$ if $n$ is the length of the shortest $S$-injective $S$-resolution of $M$.
\end{definition}\label{def-wML}

Trivially, $S$-$pd_R(M)\leq pd_R(M)$ and  $S$-$id_R(M)\leq id_R(M)$. And if $S$ is composed of units, then $S$-$pd_R(M)=pd_R(M)$.   It is also obvious that an $R$-module $M$ is $S$-projective if and only if $S$-$pd_R(M)= 0$, and is $S$-injective if and only if $S$-$id_R(M)= 0$.

\begin{lemma}\label{s-iso-pd}
Let $R$ be a ring, $S$ a multiplicative subset of $R$. If $A$ is $S$-isomorphic to $B$, then $S$-$pd_R(A) = S$-$pd_R(B)$ and $S$-$id_R(A) = S$-$id_R(B)$.
\end{lemma}
\begin{proof} We only prove  $S$-$pd_R(A)=S$-$pd_R(B)$ as the $S$-injective dimension is similar. Let $f: A\rightarrow B$ be an $S$-isomorphism. If $...\rightarrow P_n \rightarrow ...\rightarrow P_1\rightarrow P_0\xrightarrow{g} A\rightarrow 0$ is an $S$-projective resolution of $A$, then $...\rightarrow P_n \rightarrow ...\rightarrow P_1\rightarrow P_0\xrightarrow{f\circ g} B\rightarrow 0$ is an $S$-projective resolution of $B$. So $S$-$pd_R(A)\geq S$-$pd_R(B)$. Similarly we have $S$-$pd_R(B)\geq S$-$pd_R(A)$ by Proposition \ref{s-iso-inv}.
\end{proof}

\begin{proposition}\label{s-projective d}
Let $R$ be a ring and $S$ a multiplicative subset of $R$. The following statements are equivalent for an $R$-module $M$:
\begin{enumerate}
    \item $S$-$pd_R(M)\leq n$;
    \item $\Ext_R^{n+k}(M, N)$ is uniformly $S$-torsion for all  $R$-modules $N$ and all $k > 0$;
    \item $\Ext_R^{n+1}(M, N)$ is uniformly $S$-torsion for all $R$-modules $N$;
      \item if $0 \rightarrow F_n \rightarrow ...\rightarrow F_1\rightarrow F_0\rightarrow M\rightarrow 0$ is an $S$-exact sequence, where $F_0, F_1, . . . , F_{n-1}$ are $S$-projective $R$-modules, then $F_n$ is $S$-projective;
      \item if $0 \rightarrow F_n \rightarrow ...\rightarrow F_1\rightarrow F_0\rightarrow M\rightarrow 0$ is an $S$-exact sequence, where $F_0, F_1, . . . , F_{n-1}$ are projective $R$-modules, then $F_n$ is $S$-projective;
        \item if $0 \rightarrow F_n \rightarrow ...\rightarrow F_1\rightarrow F_0\rightarrow M\rightarrow 0$ is an exact sequence, where $F_0, F_1, . . . , F_{n-1}$ are $S$-projective $R$-modules, then $F_n$ is $S$-projective;
    \item if $0 \rightarrow F_n \rightarrow ...\rightarrow F_1\rightarrow F_0\rightarrow M\rightarrow 0$ is an exact sequence, where $F_0, F_1, . . . , F_{n-1}$ are projective $R$-modules, then $F_n$ is $S$-projective;
          \item there exists an $S$-exact sequence $0 \rightarrow F_n \rightarrow ...\rightarrow F_1\rightarrow F_0\rightarrow M\rightarrow 0$, where $F_0, F_1, . . . , F_{n-1}$ are projective $R$-modules and $F_n$ is $S$-projective;
    \item there exists an exact sequence $0 \rightarrow F_n \rightarrow ...\rightarrow F_1\rightarrow F_0\rightarrow M\rightarrow 0$, where $F_0, F_1, . . . , F_{n-1}$ are projective $R$-modules and $F_n$ is $S$-projective;
         \item there exists an exact sequence $0 \rightarrow F_n \rightarrow ...\rightarrow F_1\rightarrow F_0\rightarrow M\rightarrow 0$, where $F_0, F_1, . . . , F_{n}$ are $S$-projective $R$-modules.
\end{enumerate}
\end{proposition}
\begin{proof}
$(1) \Rightarrow(2)$: We prove $(2)$ by induction on $n$. For the case $n = 0$, we have $M$ is $S$-projective, then (2) holds by \cite[Theorem 2.5]{zwz21-p}. If $n>0$, then
there is an $S$-exact sequence  $0 \rightarrow F_n \rightarrow ...\rightarrow F_1\rightarrow F_0\rightarrow M\rightarrow 0$,
where each $F_i$ is  $S$-projective for $i=0,...,n$. Set $K_0 = \Ker(F_0\rightarrow M)$ and $L_0=\Im(F_1\rightarrow F_0)$. Then both
$0 \rightarrow  K_0 \rightarrow  F_0 \rightarrow  M \rightarrow  0 $ and $0 \rightarrow  F_n \rightarrow  F_{n-1} \rightarrow...\rightarrow  F_1 \rightarrow  L_0 \rightarrow  0$ are $S$-exact. Since  $S$-$pd_R(L_0)\leq n-1$ and $L_0$ is $S$-isomorphic to $K_0$, $S$-$pd_R(K_0)\leq n-1$ by Lemma \ref{s-iso-pd}. By induction, $\Ext_R^{n-1+k}(K_0, N)$ is uniformly $S$-torsion for all  $R$-modules $N$ and all $k > 0$. It follows from Corollary \ref{big-Tor} that $\Ext_R^{n+k}(M, N)$ is uniformly $S$-torsion.

$(2) \Rightarrow(3)$, $(4)\Rightarrow(5)\Rightarrow(7)$ and $(4)\Rightarrow(6)\Rightarrow(7)$:  Trivial.

$(3) \Rightarrow(4)$: Let $0 \rightarrow F_n \xrightarrow{d_n}F^{n-1} \xrightarrow{d^{n-1}}F^{n-2} ...\xrightarrow{d_2} F_1\xrightarrow{d_1} F_0\xrightarrow{d_0} M\rightarrow 0$ be an $S$-exact sequence,  where $F_0, F_1, . . . , F^{n-1}$ are $S$-projective.
Then $F_n$ is $S$-projective if and only if $\Ext^1_R(F_n,N)$ is  uniformly $S$-torsion for all  $R$-modules $N$, if and only if $\Ext^2_R(\Im(d^{n-1}),N)$ is  uniformly $S$-torsion for all  $R$-modules $N$.
Iterating these steps, we can show $F_n$ is $S$-projective if and only if $\Ext_R^{n+1}(M, N)$ is uniformly $S$-torsion for all $R$-modules $N$.

$(9)\Rightarrow(10)\Rightarrow(1)$ and $(9)\Rightarrow(8)\Rightarrow(1)$:  Trivial.

$(7)\Rightarrow(9):$ Let $...\rightarrow P_n \rightarrow P^{n-1}\xrightarrow{d^{n-1}}  P^{n-2}... \rightarrow P_0\rightarrow M\rightarrow 0$ be a projective resolution of $M$. Set $F_n=\Ker(d^{n-1})$. Then we have an exact sequence  $0\rightarrow F_n\rightarrow P^{n-1}\xrightarrow{d^{n-1}}  P^{n-2}... \rightarrow P_0\rightarrow M\rightarrow 0$. By $(7)$, $F_n$ is $S$-projective. So $(9)$ holds.

\end{proof}

Similarly, we have the following result.
\begin{proposition}\label{s-inj-d}
Let $R$ be a ring and $S$ a multiplicative subset of $R$. The following statements are equivalent for an $R$-module $M$:
\begin{enumerate}
    \item $S$-$id_R(M)\leq n$;
    \item $\Ext_R^{n+k}(N, M)$ is uniformly $S$-torsion for all  $R$-modules $N$ and all $k > 0$;
    \item $\Ext_R^{n+1}(N, M)$ is uniformly $S$-torsion for all $R$-modules $N$;
      \item if $0 \rightarrow M \rightarrow E_0\rightarrow...\rightarrow  E_{n-1}\rightarrow E_n\rightarrow 0$ is an $S$-exact sequence, where $E_0, E_1, . . . , E_{n-1}$ are $S$-injective $R$-modules, then $F_n$ is $S$-injective;
      \item if $0 \rightarrow M \rightarrow E_0\rightarrow...\rightarrow  E_{n-1}\rightarrow E_n\rightarrow 0$ is an $S$-exact sequence, where $E_0, E_1, . . . , E_{n-1}$ are injective $R$-modules, then $E_n$ is $S$-injective;
        \item if $0 \rightarrow M \rightarrow E_0\rightarrow...\rightarrow  E_{n-1}\rightarrow E_n\rightarrow 0$ is an exact sequence, where $E_0, E_1, . . . , E_{n-1}$ are $S$-injective $R$-modules, then $E_n$ is $S$-injective;
    \item if $0 \rightarrow M \rightarrow E_0\rightarrow...\rightarrow  E_{n-1}\rightarrow E_n\rightarrow 0$ is an exact sequence, where $E_0, E_1, . . . , E_{n-1}$ are injective $R$-modules, then $E_n$ is $S$-injective;
          \item there exists an $S$-exact sequence $0 \rightarrow M \rightarrow E_0\rightarrow...\rightarrow  E_{n-1}\rightarrow E_n\rightarrow 0$, where $E_0, E_1, . . . , E_{n-1}$ are injective $R$-modules and $E_n$ is $S$-injective;
    \item there exists an exact sequence $0 \rightarrow M \rightarrow E_0\rightarrow...\rightarrow  E_{n-1}\rightarrow E_n\rightarrow 0$, where $E_0, E_1, . . . , E_{n-1}$ are injective $R$-modules and $E_n$ is $S$-injective;
         \item there exists an exact sequence $0 \rightarrow M \rightarrow E_0\rightarrow...\rightarrow  E_{n-1}\rightarrow E_n\rightarrow 0$, where $E_0, E_1, . . . , E_{n}$ are $S$-injective $R$-modules.
\end{enumerate}
\end{proposition}

\begin{corollary}\label{spd-spd-1}
Let $R$ be a ring and $S'\subseteq S$  multiplicative subsets of $R$. Suppose $M$ is an $R$-module, then  $S$-$pd_R(M)\leq S'$-$pd_R(M)$ and $S$-$id_R(M)\leq S'$-$id_R(M)$.
\end{corollary}
\begin{proof} Suppose $S'\subseteq S$ are multiplicative subsets of $R$. Let  $M$ and $N$ be $R$-modules.  If  $\Ext_R^{n+1}(M, N)$ is uniformly $S'$-torsion, then  $\Ext_R^{n+1}(M, N)$ is uniformly $S$-torsion. The result follows by Proposition \ref{s-projective d}.
\end{proof}

\begin{proposition}\label{spd-exact}
Let $R$ be a ring and $S$ a multiplicative subset of $R$. Let $0\rightarrow A\rightarrow B\rightarrow C\rightarrow 0$ be an $S$-exact sequence of $R$-modules. Then the following assertions hold.
\begin{enumerate}
  \item $S$-$pd_R(C)\leq 1+\max\{S$-$pd_R(A),S$-$pd_R(B)\}$.
    \item  If $S$-$pd_R(B)<S$-$pd_R(C)$, then $S$-$pd_R(A)=S$-$pd_R(C)-1>S$-$pd_R(B).$
        \item     $S$-$id_R(A)\leq 1+\max\{S$-$id_R(B),S$-$id_R(C)\}$.
            \item  If $S$-$id_R(B)<S$-$id_R(A)$, then $S$-$id_R(C)=S$-$id_R(A)-1>S$-$id_R(B).$
\end{enumerate}
\end{proposition}
\begin{proof} The proof  is similar with that of the classical case (see \cite[Theorem 3.5.6]{fk16} and \cite[Theorem 3.5.13]{fk16}). So we omit it.
\end{proof}

\begin{proposition}\label{spd-sp-exact}
 Let $0\rightarrow A\rightarrow B\rightarrow C\rightarrow 0$ be an $S$-split $S$-exact sequence of $R$-modules. Then the following assertions hold.
\begin{enumerate}
  \item $S$-$pd_R(B)=\max\{S$-$pd_R(A),S$-$pd_R(C)\}$.
  \item $S$-$id_R(B)=\max\{S$-$id_R(A),S$-$id_R(C)\}$.
\end{enumerate}
\end{proposition}
\begin{proof} We only show the first assertion since the other one is similar.  Since the $S$-projective dimensions of $R$-modules are invariant under $S$-isomorphisms by Lemma \ref{s-iso-pd}, we may assume $0\rightarrow A\xrightarrow{f} B\xrightarrow{g} C\rightarrow 0$ is an $S$-split exact sequence. So there exists $R$-homomorphisms $f': B\rightarrow A$ and $g': C\rightarrow B$ such that $f'\circ f=s_1\Id_A$ and  $g\circ g'=s_2\Id_C$ for some $s_1,s_2\in S$. To prove $(1)$,
we just need to show that $0\rightarrow\Ext_R^n(M,A)\xrightarrow{\Ext_R^n(M,f)} \Ext_R^n(M,B)\xrightarrow{\Ext_R^n(M,g)} \Ext_R^n(M,C)\rightarrow 0$ is an $S$-exact sequence for any $R$-module $M$.
Since  the composition map $\Ext_R^n(M,f') \circ \Ext_R^n(M,f): \Ext_R^n(M,A) \rightarrow \Ext_R^n(M,A)$ is equal to
$\Ext_R^n(M,s_1\Id_A)$ which is just the multiplication map by $s_1$, we have $\Ext_R^n(M,f)$ is an $S$-split $S$-monomorphism. Similarly, $\Ext_R^n(M,g)$ is an $S$-split $S$-epimorphism.
\end{proof}

Let $\p$ be a prime ideal of $R$ and $M$ an $R$-module. Denote $\p$-$pd_R(M)$ (resp., $\p$-$id_R(M)$) to be $(R-\p)$-$pd_R(M)$ (resp.,  $(R-\p)$-$id_R(M)$) briefly. The next result gives a new local characterization of projective dimension and injective dimension of an $R$-module.

\begin{proposition}\label{pd-spd}
Let $R$ be a ring and $M$ an $R$-module. Then
\begin{align*}
pd_R(M)=\sup\{\p\mbox{-}pd_R(M)|\p\in \Spec(R)\}=\sup\{\m\mbox{-}pd_R(M)|\m\in \Max(R)\} .
\end{align*}
and
\begin{align*}
id_R(M)=\sup\{\p\mbox{-}id_R(M)|\p\in \Spec(R)\}=\sup\{\m\mbox{-}id_R(M)|\m\in \Max(R)\} .
\end{align*}
\end{proposition}
\begin{proof} We only show the first equation since the other one is similar. Trivially, $\sup  \{\m\mbox{-}pd_R(M) |\ \m\in \Max(R)\} \leq \sup \{\p\mbox{-}pd_R(M) |\ \p\in \Spec(R)\} \leq pd_R(M)$.
Suppose $\sup  \{\m\mbox{-}pd_R(M) |\ \m\in \Max(R)\}=n$. For any $R$-module $N$, there exists an element $s^{\m}\in R-\m$ such that  $s^{\m}\Ext_R^{n+1}(M, N)=0$  by Proposition \ref{s-projective d}. Since the ideal generated by all $s^{\m}$ is $R$, we have $\Ext_R^{n+1}(M, N)=0$ for all $R$-modules $N$. So $pd_R(M)\leq n$. Suppose $\sup  \{\m\mbox{-}pd_R(M) |\ \m\in \Max(R)\}=\infty$. Then for any $n\geq 0$, there exists a maximal ideal $\m$ and an element $s^{\m}\in R-\m$ such that  $s^{\m}\Ext_R^{n+1}(M, N)\not=0$ for some $R$-module $N$. So for any $n\geq 0$, we have $\Ext_R^{n+1}(M, N)\not=0$ for some $R$-module $N$. Thus $pd_R(M)=\infty$. So the equalities hold.
\end{proof}

\section{On the $S$-global dimensions of rings}
Recall that the global dimension gl.dim$(R)$ of a ring $R$ is the supremum of projective dimensions of all $R$-modules (see \cite[Definition 3.5.17]{fk16}). Now, we introduce the $S$-analogue of global dimensions of rings $R$ for a multiplicative subset $S$ of $R$.

\begin{definition}\label{w-phi-projective }
The $S$-global dimension of a ring $R$ is defined by
\begin{center}
$S$-gl.dim$(R) = \sup\{S$-$pd_R(M) | M $ is an $R$-module$\}.$
\end{center}
\end{definition}\label{def-wML}
Obviously, $S$-gl.dim$(R)\leq $gl.dim$(R)$ for any multiplicative subset $S$ of $R$.  And if $S$ is composed of units, then $S$-gl.dim$(R)=$gl.dim$(R)$ .  The next result characterizes the $S$-global dimension of a ring $R$.
\begin{proposition}\label{w-g-projective}
Let $R$ be a ring and $S$ a multiplicative subset of $R$. The following statements are equivalent for $R$:
\begin{enumerate}
  \item  $S$-gl.dim$(R)\leq  n$;
    \item  $S$-$pd_R(M)\leq n$ for all $R$-modules $M$;
    \item $\Ext_R^{n+k}(M, N)$ is uniformly $S$-torsion for all $R$-modules $M, N$ and all $k > 0$;
    \item  $\Ext_R^{n+1}(M, N)$ is uniformly $S$-torsion for all $R$-modules $M, N$;
    \item  $S$-$id_R(M)\leq n$ for all $R$-modules $M$.
\end{enumerate}
\end{proposition}
\begin{proof}
$(1) \Rightarrow  (2)$ and $(3)\Rightarrow  (4)$: Trivial

$(2) \Rightarrow (3)$ and $ (5)\Rightarrow  (3) $: Follows from Proposition \ref{s-projective d}.

$(4) \Rightarrow  (2)$: Let $M$ be an $R$-module and $0 \rightarrow F_n \rightarrow ...\rightarrow F_1\rightarrow F_0\rightarrow M\rightarrow 0$ an exact sequence, where $F_0, F_1, . . . , F^{n-1}$ are projective $R$-modules.
To complete the proof, it suffices, by Proposition \ref{s-projective d}, to prove that $F_n$ is
$S$-projective. Let $N$ be an $R$-module.
Thus $S$-$pd_R(N)\leq n$ by (4). It follows from Corollary \ref{big-Tor} that $\Ext_R^1 (N, F_n)\cong \Ext_R^{n+1}(N, M)$ is uniformly $S$-torsion. Thus $F_n$ is $S$-projective.

$(4) \Rightarrow  (5)$: Let $M$ be an $R$-module and $0 \rightarrow M \rightarrow E_0\rightarrow...\rightarrow  E_{n-1}\rightarrow E_n\rightarrow 0$ an exact sequence with $E_0, E_1, . . . , E_{n-1}$ are injective $R$-modules. By dimension shifting, we have  $\Ext_R^{n+1}(M, N)\cong \Ext_R^{1}( E_n, N)$. So $\Ext_R^{1}(E_n, N)$ is uniformly $S$-torsion for any $R$-module $N$. Thus $E_n$ is $S$-injective by \cite[Theorem 4.3]{QKWCZ21} . Consequently,  $S$-$id_R(M)\leq n$ by Theorem \ref{s-inj-d}.
\end{proof}

Consequently, we have $S$-gl.dim$(R) = \sup\{S$-$pd_R(M) | M $ is an $R$-module$\}= \sup\{S$-$id_R(M) | M $ is an $R$-module$\}.$

Let $\p$ be a prime ideal of a ring $R$ and $\p$-$gl.dim(R)$ denote $(R-\p)$-$gl.dim(R)$ briefly. By Proposition \ref{pd-spd}, we have a new local characterization of  global dimensions of commutative rings.
\begin{corollary}\label{wgld-swgld}
Let $R$ be a ring. Then
\begin{align*}
&gl.dim(R)=\sup\{\p\mbox{-}gl.dim(R)|\p\in \Spec(R)\}=\sup\{\m\mbox{-}gl.dim(R)|\m\in \Max(R)\} .
\end{align*}
\end{corollary}

Recall from \cite{zwz21-p} that an $R$-module $M$ is called $S$-semisimple provided that any $S$-short exact sequence $0\rightarrow A\rightarrow M\rightarrow C\rightarrow 0$ is $S$-split. And $R$ is called an $S$-semisimple ring provided that any free  $R$-module is $S$-semisimple. Thus by \cite[Theorem 3.5]{zwz21-p}, the following result holds.
\begin{corollary}\label{s-vn-ext-char}
Let $R$ be a ring and $S$ a multiplicative subset of $R$. The following assertions are equivalent:
\begin{enumerate}
\item $R$ is an $S$-semisimple ring;
\item every $R$-module is $S$-semisimple;
\item every $R$-module is $S$-projective;
\item  every $R$-module is $S$-injective;
\item $R$ is uniformly $S$-Noetherian and  $S$-von Neumann regular;
\item  there exists an element $s\in S$ such that for any ideal $I$ of $R$ there is an $R$-homomorphism $f_I: R\rightarrow I$ satisfying $f_I(i)=si$ for any $i\in I$.
\item  $S$-gl.dim$(R)=0$.
\end{enumerate}
\end{corollary}

The following example shows that the global dimension of rings and the $S$-global dimension of rings can are be wildly different.
\begin{example} Let $T=\mathbb{Z}_2\times \mathbb{Z}_2$ be a semi-simple ring and $s=(1,0)\in T$. Let $R=T[x]/\langle sx,x^2\rangle$ with $x$ the indeterminate  and $S=\{1,s\}$ be a multiplicative subset of $R$. Then $S$-gl.dim$(R)=0$ by \cite[Theorem 3.5]{zwz21-p}. Since $R$ is a non-reduced noetherian ring, gl.dim$(R)=\infty$ by \cite[Corollary 4.2.4]{g}.
\end{example}

\section{$S$-global dimensions of factor rings and  polynomial rings}

In this section, we mainly consider the $S$-global dimensions of factor rings and  polynomial rings. Firstly, we give an inequality of $S$-global dimensions for ring homomorphisms. Let $\theta:R\rightarrow T$ be a ring homomorphism. Suppose $S$ is a multiplicative subset of $R$, then $\theta(S)=\{\theta(s)|s\in S\}$  is a multiplicative subset of $T$.

\begin{proposition}\label{spd-changring}
Let $\theta:R\rightarrow T$ be a ring homomorphism, $S$ a multiplicative subset of $R$. Suppose $M$ is an $T$-module. Then
\begin{center}
$S$-$pd_{R}(M)\leq\theta(S)$-$pd_{T}(M)+S$-$pd_{R}(T).$
\end{center}
\end{proposition}
\begin{proof} Assume $\theta(S)$-$pd_{T}(M)=n<\infty$. If $n=0$, then $M$ is  $\theta(S)$-projective  over $T$.  Then there exists $\theta(S)$-split short exact sequence $0\rightarrow A\rightarrow F\rightarrow M\rightarrow 0$ with $F$ a free $R$-module of rank $\geq 1$.
By Proposition \ref{spd-sp-exact}, we have $\theta(S)$-$pd_{T}(F)\geq \theta(S)$-$pd_{T}(M)$. So  $S$-$pd_{R}(M)\leq S$-$pd_{R}(F)=S$-$pd_{R}(T)\leq n+S$-$pd_{R}(T)$.

Now we assume $n>0$. Let $0\rightarrow A\rightarrow F\rightarrow M\rightarrow 0$ be an exact sequence of $T$-modules, where $F$ is a free $T$-module of rank $\geq 1$. Then $\theta(S)$-$pd_{T}(A)=n-1$ by Corollary \ref{big-Tor} and Proposition \ref{s-projective d}. By induction, $S$-$pd_{R}(A)\leq n-1+S$-$pd_{R}(T)$. Note that $S$-$pd_{R}(T)=S$-$pd_{R}(F)$. By Proposition \ref{spd-exact}, we have
\begin{align*}
 S\mbox{-}pd_{R}(M)&\leq 1+\max\{S\mbox{-}pd_{R}(F), S\mbox{-}pd_{R}(A)\} \\
  &\leq 1+n-1+S\mbox{-}pd_{R}(T) \\
   &=\theta(S)\mbox{-}pd_{T}(M)+S\mbox{-}pd_{R}(T).
\end{align*}
\end{proof}

Let $R$ be a ring, $I$ an ideal of $R$ and  $S$  a multiplicative subset of $R$. Then $\pi:R\rightarrow R/I$ is a ring epimorphism and $\pi(S):=\overline{S}=\{s+I\in R/I|s\in S\}$ is naturally a multiplicative subset of $R/I$.

\begin{proposition}\label{s-pd-poly-3}
Let $R$ be a ring, $S$ a multiplicative subset of $R$. Let $a$ be a non-zero-divisor  in $R$ which does not divide any element in $S$. Written $\overline{R}=R/aR$ and $\overline{S}=\{s+aR\in \overline{R}|s\in S\}$. Then the following assertions hold.
\begin{enumerate}
\item  Let $M$ be a nonzero $\overline{R}$-module. If  $\overline{S}$-$pd_{\overline{R}}(M)<\infty$, then
\begin{center}
$S$-$pd_{R}(M)=\overline{S}$-$pd_{\overline{R}}(M)+1.$
\end{center}
\item   If $\overline{S}$-gl.dim$(\overline{R})<\infty$, then
\begin{center}
$S$-gl.dim$(R) \geq \overline{S}$-gl.dim$(\overline{R})+1$.
\end{center}
\end{enumerate}
\end{proposition}
\begin{proof} (1) Set $\overline{S}$-$pd_{\overline{R}}(M)=n$. Since  $a$ is a non-zero-divisor   which does not divide any element in $S$, then the exact sequence $0\rightarrow aR\rightarrow R\rightarrow R/aR\rightarrow 0$ does not $S$-split. Thus $S$-$pd_R(\overline{R})=1$. By Proposition \ref{spd-changring}, we have $S$-$pd_R(M)\leq \overline{S}$-$pd_{\overline{R}}(M)+1=n+1.$ Since  $\overline{S}$-$pd_{\overline{R}}(M)=n$, then there is an injective $\overline{R}$-module $C$ such that $\Ext^n_{\overline{R}}(M,C)$ is not uniformly $\overline{S}$-torsion. By \cite[Theorem 2.4.22]{fk16}, there is an injective $R$-module $E$ such that $0\rightarrow C\rightarrow E\rightarrow E\rightarrow 0$ is exact. By \cite[Proposition 3.8.12(4)]{fk16}, $\Ext^{n+1}_{R}(M,E)\cong \Ext^n_{\overline{R}}(M,C)$. Thus $\Ext^{n+1}_{R}(M,E)$ is not uniformly $S$-torsion. So $S$-$pd_{R}(M)=\overline{S}$-$pd_{\overline{R}}(M)+1$.

(2) Let $n=\overline{S}$-gl.dim$(\overline{R})$. Then there is a nonzero $\overline{R}$-module $M$ such that $\overline{S}$-$pd_{\overline{R}}(M)=n$. Thus $S$-$pd_{R}(M)=n+1$ by (1). So $S$-gl.dim$(R) \geq \overline{S}$-gl.dim$(\overline{R})+1$.
\end{proof}

Let $R$ be a ring and $M$ an $R$-module. $R[x]$ denotes the  polynomial ring with one indeterminate, where all coefficients are in $R$. Set $M[x]=M\otimes_RR[x]$, then $M[x]$ can be seen as an $R[x]$-module naturally. It is well-known gl.dim$(R[x])=$gl.dim$(R)$ (see \cite[Theorem 3.8.23]{fk16}). In this section, we give a $S$-analogue of this result. Let $S$ be a multiplicative subset of $R$, then $S$ is a multiplicative subset of $R[x]$ naturally.

\begin{lemma}\label{tor-s-poly}
Let $R$ be a ring, $S$ a multiplicative subset of $R$. Suppose $T$ is an $R$-module and  $F$ is an $R[x]$-module.   If  $P$ is $S$-projective over $R[x]$, then $P$ is  $S$-projective over $R$.

\end{lemma}
\begin{proof}Suppose $P$ is an $S$-projective $R[x]$-module. Then there exists a free  $R[x]$-module $F$ and a  $S$-split $R[x]$-short exact sequence $0\rightarrow K\rightarrow F\xrightarrow{\pi} P\rightarrow 0$. Thus we have an $R[x]$-homomorphism $\pi':P\rightarrow F$ such that $\pi\circ\pi'=s\Id_P$ for some $s\in S$. Note that $\pi'$ is also an  $R$-homomorphism.  So  $0\rightarrow K\rightarrow F\xrightarrow{\pi} P\rightarrow 0$ is also $S$-split over $R$. Note that $F$ is also a  free $R$-module. So $P$ is  $S$-projective over $R$ by \cite[Proposition 2.8]{zwz21-p}.

\end{proof}

\begin{proposition}\label{s-pd-poly}
Let $R$ be a ring, $S$ a multiplicative subset of $R$ and $M$ an $R$-module. Then $S$-$pd_{R[x]}(M[x])=S$-$pd_R(M)$.
\end{proposition}
\begin{proof}

Assume that $S$-$pd_R(M)\leq n$.  Then $M$ has an $S$-projective resolution over $R$:
$$0\rightarrow P_n\rightarrow\cdots\rightarrow P_1\rightarrow P_0\rightarrow M\rightarrow0.$$
Since $R[x]$ is free over $R$, $R[x]$ is an $S$-flat $R$-module by \cite[Proposition 2.7]{zwz21-p}. Thus the natural sequence $$0\rightarrow P_n[x]\rightarrow\cdots\rightarrow P_1[x]\rightarrow P_0[x]\rightarrow M[x]\rightarrow0$$ is $S$-exact over $R[x]$. Consequently, $S$-$pd_{R[x]}(M[x])\leq n$ by Proposition \ref{s-projective d}.

Let $0\rightarrow F_n \rightarrow ...\rightarrow F_1\rightarrow F_0\rightarrow M[x]\rightarrow 0$ be an exact sequence with each $F_i$ $S$-projective over $R[x]$ ($1\leq i\leq n$). Then it is also $S$-projective resolution of $M[x]$ over $R$  by Lemma \ref{tor-s-poly}. Thus $\Ext^{n+1}_R(M[x],N)$ is  uniformly $S$-torsion for any $R$-module $N$ by Proposition \ref{s-projective d}. It follows that  $s\Ext^{n+1}_R(M[x],N)=s \prod\limits_{i=1}^{\infty} \Ext^{n+1}_R(M,N)=0 $. Thus $\Ext^{n+1}_R(M,N)$ is  uniformly $S$-torsion. Consequently, $S$-$pd_R(M)\leq S$-$pd_{R[x]}(M[x])$ by Proposition \ref{s-projective d} again.
\end{proof}

Let $M$ be an $R[x]$-module then $M$ can be naturally viewed as an $R$-module. Define $\psi:M[x]\rightarrow M$ by $$
\psi(\sum\limits_{i=0}^nx^i\otimes m_i)=\sum\limits_{i=0}^nx^i m_i,\qquad m_i\in M.$$ And define $\varphi:M[x]\rightarrow M[x]$ by $$\varphi(\sum\limits_{i=0}^nx^i\otimes m_i)=\sum\limits_{i=0}^nx^{i+1}\otimes m_i-\sum\limits_{i=0}^nx^i\otimes xm_i,\qquad m_i\in M.$$

\begin{lemma}\cite[Theorem 3.8.22]{fk16}\label{exact-s-poly}
Let $R$ be a ring, $S$ a multiplicative subset of $R$. For any $R[x]$-module $M$, $$0\rightarrow M[x]\xrightarrow{\varphi} M[x]\xrightarrow{\psi} M\rightarrow 0$$
is exact.
\end{lemma}

\begin{theorem}\label{s-wgd-poly}
Let $R$ be a ring, $S$ a multiplicative subset of $R$. Then $S$-gl.dim$(R[x])=S$-gl.dim$(R)+1$.
\end{theorem}
\begin{proof} Let $M$ be an $R[x]$-module. Then, by Lemma \ref{exact-s-poly},  there is an exact sequence over $R[x]$:
$$0\rightarrow M[x]\rightarrow M[x]\rightarrow M\rightarrow 0.$$
By Proposition \ref{spd-exact}, Proposition \ref{spd-changring} and Proposition \ref{s-pd-poly},
\begin{center}
 $S$-$pd_R(M)\leq S$-$pd_{R[x]}(M)\leq 1+S$-$pd_{R[x]}(M[x])=1+S$-$pd_R(M)\qquad \qquad (\ast)$.
\end{center}
Thus if $S$-gl.dim$(R)< \infty$, then $S$-gl.dim$(R[x])<\infty$.

Conversely, if $S$-gl.dim$(R[x])< \infty$, then for any $R$-module $M$, $S$-$pd_R(M)=S$-$pd_{R[x]}(M[x])< \infty$ by Proposition \ref{s-pd-poly}. Therefore we have $S$-gl.dim$(R)< \infty$ if and only if $S$-gl.dim$(R[x])< \infty$.
Now we assume that both of these are finite. Then $S$-gl.dim$(R[x])\leq S$-gl.dim$(R)+1$ by $(\ast)$. Since $R\cong R[x]/xR[x]$, $S$-gl.dim$(R[x]) \geq S$-gl.dim$(R)+1$ by Proposition \ref{s-pd-poly-3}. Consequently, we have $S$-gl.dim$(R[x])=S$-gl.dim$(R)+1$.
\end{proof}
\begin{corollary}\label{s-wgd-poly-duo}
Let $R$ be a ring, $S$ a multiplicative subset of $R$. Then for any $n\geq 1$ we have
\begin{center}
$S$-gl.dim$(R[x_1,...,x_n])=S$-gl.dim$(R)+n$.
\end{center}
\end{corollary}

\begin{acknowledgement}\quad\\
The author was supported by the Natural Science Foundation of Chengdu Aeronautic Polytechnic (No. 062026) and the National Natural Science Foundation of China (No. 12061001).
\end{acknowledgement}

https://arxiv.org/abs/2106.10441.

\end{document}